\documentclass[11pt,reqno]{amsart}
\usepackage{amsmath,amsthm,amscd,amsfonts,amssymb,color}
\usepackage{cite}
\usepackage[bookmarksnumbered,colorlinks,plainpages]{hyperref}
\usepackage{graphicx} 

\usepackage{amsmath, amssymb, amsthm}
\usepackage{booktabs}
\usepackage{array}
\usepackage{longtable}

\newtheorem{theorem}{Theorem}[section]

\newtheorem{corollary}[theorem]{Corollary}
\theoremstyle{definition}
\newtheorem{definition}[theorem]{Definition}
\newtheorem{example}[theorem]{Example}
\theoremstyle{remark}
\newtheorem{remark}[theorem]{Remark}
\numberwithin{equation}{section}
\def\DJ{\leavevmode\setbox0=\hbox{D}\kern0pt\rlap
	{\kern.04em\raise.188\ht0\hbox{-}}D}
\begin{document}
	\title{\bf{A Unified Approach to Non-Unique Fixed Points via Modified C-Class ĆiriC-Type Contractions}}
	\author[O. J. Omidire]{O. J. Omidire}
	\address{\textbf{O. J. Omidire}
		Department of Mathematical Sciences,
		Osun State University, Osogbo, Nigeria.}
	\email{olaoluwa.omidire@uniosun.edu.ng, omidireolaoluwa@gmail.com}
\author[M. O. Olatinwo]{M. O. Olatinwo}
\address{\textbf{M. O. Olatinwo}
	Department of Mathematics,
	Obafemi Awolowo University, Ile Ife.}
\email{memudu.olatinwo@gmail.com}
	\author[K. R. Tijani]{K. R. Tijani}
\address{\textbf{K. R. Tijani}
	Department of Mathematical Sciences,
	Osun State University, Osogbo, Nigeria.}
\email{kamilu.tijani@uniosun.edu.ng}
\author[B. T. Ishola]{B. T. Ishola}
\address{\textbf{B. T. Ishola}
	Department of Mathematics,
	Obafemi Awolowo University, Ile Ife.}
\email{isholatunde@gmail.com}
\begin{abstract}
	Fixed point theory is central to nonlinear analysis, yet classical results often assume strict contractivity and guarantee uniqueness, limiting applicability to mappings with inherently non-unique fixed points. This paper addresses this gap by introducing a modified C-class function and establishing new non-unique fixed point theorems for Ćirić-type mappings in complete metric spaces.
	
	We construct a generalized contractive inequality using a three-variable modified C-class function combined with an altering distance function. Under orbital continuity and orbital completeness, we prove that Picard iteration converges to a fixed point without requiring global Lipschitz conditions or strict contractivity. Carefully designed examples illustrate the sharpness and generality of the framework, strictly extending classical Ćirić and Chatterjea contractions as well as hybrid $(\Phi,\psi)$-contractions.
	
	Our main contributions are: (i) a unified framework encompassing classical and generalized C-class contractions, (ii) existence results for non-unique fixed points, (iii) a simple criterion for uniqueness when desired, and (iv) practical validation through nontrivial examples. These results provide a robust foundation for further research in nonlinear operator theory, iterative methods, and applications in optimization and variational problems.
\end{abstract}
	\maketitle
	\textbf{Keywords/phrases:} Non-unique fixed point, C-class function, Ciric-type mapping.	~~\\
	MSC2020 Mathematics Subject Classification:47H10; 47H09\\
	\section{Introduction}
	The development of fixed point theory has evolved from classical contraction mapping principles  of Banach (see \cite{Banach1922}) toward increasingly flexible contractive frameworks capable of capturing nonlinear phenomena beyond strict Banach-type conditions. A decisive shift occurred with the introduction of generalized contraction schemes that relax linear Lipschitz-type constraints and incorporate hybrid or functional control conditions.
	
	A foundational contribution in this direction was made by Rakotch 1962 \cite{Rakotch1962} when he introduced a monotone decreasing function $\alpha: (0, \infty) \to [0,1)$ such that, for each $w,s \in G, ~~w \neq s,$ 
	\begin{equation}
		\delta(Zw,Zs) \leq \alpha \Big( \delta(w,s) \Big),
	\end{equation}
	where $Z$ is an operator on $G.$\\
	This was followed by  Chatterjea \cite{Chatterjea1972} who introduced generalization of contraction mappings, now popularly referred to as Chatterjea contractions. His condition, involved dual conditions of the form
	\begin{equation}
		\delta(Zw,Zs) \leq a\Big( \delta(w,Zs)+\delta(s,Zw)\Big)~~~~~a \in [0,\frac{1}{2})
	\end{equation}
	which further expanded the class of admissible nonlinear operators.  Ciric \cite{Ciric1974a}, introduced what is now known as Ciric-type quasi-contractions. A mapping $Z: G \to G$ was referred to as quasi-contraction iff 
	\begin{equation}\label{eqn1.3}
		\delta(Zw, Zs)\leq q. max\{\delta(w, s); \delta(w, Zw); \delta(s, Zs); \delta(w, Zs); d(s, Zw)\}
	\end{equation}
	for some $q < 1$ and for all $w,s \in G.$
	This mapping extended the Banach contraction principle by replacing the single-distance condition with a maximum involving multiple interpoint distances. This broader contractive inequality provides framework for solving nonlinear operators not covered by Banach contraction and some of its generalizations.\\
	Shortly thereafter, in \cite{Ciric1974}, Ciric demonstrated that under certain contractive structures, fixed point (of some non-linear operator equations and related problems, if exists) may not be unique, thereby initiating systematic study of non-unique fixed point phenomena within generalized contractive settings.\\
	He considered operator $Z$ on $G$ which are not necessarily continuous and which satisfy the condition
	\begin{equation}
		min\{\delta(Zw,Zs), \delta(w,Zw), \delta(s,Zs),\}-min\{\delta(w,Zs), \delta(s,Zw)\} \leq q.\delta(w,s).
	\end{equation}

	The study of non-unique fixed points gained independent momentum with the realization that multiplicity naturally arises in nonlinear operator equations, differential inclusions, and optimization theory. Ciric’s 1974 work \cite{Ciric1974} formally established that suitably weakened contraction inequalities may guarantee existence without uniqueness, thereby separating the two properties conceptually. This distinction has since become central in modern nonlinear analysis.
	
	In subsequent decades, research shifted toward functional and hybrid contractive conditions, replacing scalar contraction constants with auxiliary control functions. A notable advancement in this direction was the introduction of C-class functions by Ansari (2014) in his note  \cite{Ansari2014}, Ansari introduced the concept of $C-$class functions.	
	\begin{definition}
		\label{def:1.1} \cite{Ansari2014} A mapping $k:[0,\infty )^{2}\rightarrow \mathbb{R}
		$ is called \textit{$C-$class} function if it is continuous and satisfies
		the following axioms: for all $w,s~ \in \lbrack
		0,\infty )$
		
		(1) $k(w,s)\leq w$;
		
		(2) $k(w,s)=w$ implies that either $w=0$ or $s=0.$
	\end{definition}
	
	\noindent	\textbf{Note:} we have that $k(0,0)=0$.
	\newline
	We denote the set of $C-$class functions by $\mathcal{C}$.
	See \cite{Ansari2014} for examples of C-class function.
	Ansari’s C-class framework provided a flexible functional inequality scheme capable of unifying several earlier altering distance and weak contraction models. By embedding contractivity via a function $k$ satisfying structural positivity and monotonicity properties. C-class mappings offered a powerful abstract platform for generalizing both classical and hybrid contractions.
	
	Olatinwo  \cite{Olatinwo2019} developed non-unique fixed point theorems of Ćirić-type for $(\Phi,\psi)-$hybrid contractive mappings. He employed the following contractivity conditions: \cite{Olatinwo2019} Let $(G,\delta)$ be a metric space.\\
	(a) For a mapping $Z:G \to G,~~~\exists$ functions $u,v,w: \mathbb{R}_+ \to \mathbb{R}_+$ with $v(0)=0=w(0),$ and functions $\beta: \mathbb{R}^5_+ \to \mathbb{R}_+,~~~\Phi, \psi: \mathbb{R}_+ \to \mathbb{R}_+$ such that
	\begin{eqnarray*}
		\Phi(M) &\leq& \beta\Big(\delta(x,y), \delta(x,Zx), \delta(y,Zy),\\
		&& u(\delta(x,Zx))v(\delta(y,Zx))\delta(x,Zy), \delta(y,Zx)u(\delta(x,Zx))\Big),
	\end{eqnarray*}
	$\forall~~x,y~\in G$ such that $M \geq 0,$ where
	\begin{equation}
		M = min\{\delta(Zx,Zy), \delta(x,Zx), \delta(y,Zy),\}-w\Big(min\{\delta(x,Zy), \delta(y,Zx)\}\Big),
	\end{equation}
	and\\
	(b) For a mapping $Z:G \to G,~~~\exists$ continuous functions $u,v,w: \mathbb{R}_+ \to \mathbb{R}_+$ with $v(0)=0=w(0),$ and functions $\beta: \mathbb{R}^5_+ \to \mathbb{R}_+,~~~\Phi, \psi: \mathbb{R}_+ \to \mathbb{R}_+$ such that
	\begin{eqnarray*}
		\Phi(N) &\leq& \beta\Big(\delta(x,y), \delta(x,Zx), \delta(y,Zy),\\
		&& u(\delta(x,Zx))v(\delta(y,Zx))\delta(x,Zy), \delta(y,Zx)u(\delta(x,Zx))\Big),
	\end{eqnarray*}
	$\forall~~x,y~\in G$ such that $N \geq 0,$ where
	\begin{equation}
		N = min\{\delta(Zx,Zy), max\{\delta(x,Zx), \delta(y,Zy)\}\}-w\Big(min\{\delta(x,Zy), \delta(y,Zx)\}\Big),
	\end{equation}
	the functions $\beta, \Phi, \psi$ satisfy certain conditions. 
	His work extended earlier hybrid contraction results by incorporating multiple control functions and adapting techniques inspired by A-contractions of Akram et al. \cite{Akram2008}. Olatinwo’s results generalized several known Ćirić-type and Chatterjea-type theorems, demonstrating that hybrid functional inequalities can preserve multiplicity while guarantee existence of fixed point under weakened assumptions.
	
	More recently, Omidire et al. \cite{Omidire2025} investigated fixed point of generalized C-class contractivity conditions, emphasizing constructive approaches to fixed point approximation.
	\begin{definition}\cite{Omidire2025}
		\label{def:3.1} Let $G$ be a metric space. A function $Z:G\rightarrow G$
		was called a \textbf{C-Class Akram Contraction} if $\forall ~w,s\in G$, 
		\begin{equation}
			\delta(Zw,Zs)\leq A(\delta(w,s),\delta(w,Zw),\delta(s,Zs))  \label{eqn1.7}
		\end{equation}%
		and some $(\psi ,\phi ,G)\in \Psi \times \Phi \times \mathcal{C}$ where $A:%
		\mathbb{R}_{+}^{3}\rightarrow \mathbb{R}^2_{+}$ is a set of functions
		satisfying:
		
		\begin{enumerate}
			\item $A$ is continuous on the set $\mathbb{R}_{+}^{3}$ (with respect to the
			Euclidean metric on $\mathbb{R}^{3});$ and
			
			\item If any of the conditions $a\leq A(a,b,b)$ or $a\leq A(b,b,a)$ or $%
			a\leq A(b,a,b)$ holds for some $a,b\in \mathbb{R}_{+}$, then $\exists ~G\in 
			\mathcal{C}$ such that $\psi (a)\leq G(\psi (b),\phi (b))$.
		\end{enumerate}
	\end{definition}
	While primarily focused on convergence analysis of iterative processes, their framework highlighted the structural robustness of C-class inequalities and their capacity to unify disparate contraction models under a single functional umbrella.
	
	Despite these advances, the intersection of Ćirić-type quasi-contractive structures with C-class functional control in the specific context of non-unique fixed points has not been explored, to the best of our knowledge. Moreover, many results impose auxiliary conditions such as continuity that restrict applicability.
	This study therefore, focuses on embeding Ciric-type (minimum)maximum-distance structures directly into a C-class functional inequality framework and systematically analyze non-unique fixed point phenomena under this generalized scheme.\\
	We provide a unified theorem that simultaneously recovers: classical Ćirić contractions, Chatterjea contractions, $(\Phi,\psi)-$hybrid contractive mappings, generalized C-class contractive conditions (C-class Akram Contraction).
	
	Thus, the present work does not merely extend isolated results but consolidates classical quasi-contractive theory and modern functional contractivity into a coherent non-unique fixed point framework.\\
	~~\\ 
	\par The following definitions shall be required in the sequel:\\
	Given a complete metric space $(M, \delta)$ with a self map operator $Z$ on $M,$ and that $Fix(Z) = \{w \in M : Zw = w\}$ is the set of fixed points of $Z.$
	\begin{definition}\cite{Ciric1974a,Olatinwo2019}
	$O(w,Z)$ is referred to as the orbit of $Z$ at $w$ if
	$$O(w,Z)= \{w,Zw,Z^2w, Z^3w, \cdots, Z^nw, \cdots\}.$$	
	\end{definition}
	\begin{definition}\cite{Ciric1974a,Olatinwo2019}
	A metric space $(M,\delta)$ is $Z-$orbitally complete if $Z$ is a self map on $M$ and any Cauchy sub-sequence $\{Z^{n_i}w\}$ in the orbit $O(w,Z)~~\forall~~ w \in M,$ converges to a point in $M.$ 	
	\end{definition}
	\begin{definition}\cite{Olatinwo2019}
	A self map on $M$ is orbitally continuous if 
	$$\lim_{i \to \infty} \delta(Z^{n_i}w,w^*)=0 \implies \lim_{i \to \infty} \delta(Z(Z^{n_i}),Zw^*)=0.$$
\end{definition}
\begin{definition}\cite{Ansari2014}
	A function $\beth: [0,\infty) \to [0,\infty)$ is said to be an altering distance function if it is continuous, non-decreasing and satisfies
	$$\beth(x)=0 \iff x=0,~~~x \in [0,\infty).$$
\end{definition}
	\section{Preliminary results}
	We introduce the following definitions which are generalization (modifications) to C-class function.\\ 
	\begin{definition}\label{defn1}
		A function $k:[0, \infty)^3 \to \mathbb{R}$ will be called a modified C-class function, if is continuous and for all $w, s~\in \mathbb{R} $  the following hold:\\
		\begin{itemize}
			\item $k(w,s,w) = k(w,w,s) \leq w;$
			\item  $k(w,w,w)\leq w;$ and
			\item $k(w,s,w)=w \implies w=0$ or $s=0$
			
		\end{itemize} 
	\end{definition}
	Note: $k(0,0,0)=0.$
	\begin{example}
		Let $w,s~\in [0, \infty),$ then if:
		\begin{enumerate}
			\item $k(w,s,w) := w-sw,~~~~~~k(w,s,w) = w~~~~~~\implies ~~s =0.$
			\item $k(w,s,w):= \frac{w -sw}{1+s},~~~~~~k(w,s,w) = w~~~~\implies ~~s =0.$
			\item $k(w,s,w):= \frac{w}{1+sw},~~~~~k(w,s,w) = w~~~\implies~~s =0.$
		\end{enumerate}
	\end{example}
	\begin{definition}\label{defn2}
		Let $M$ be a metric space. A Mapping $Z: M \to M$ will be called generalized  \textbf{C-class Ciric contraction} if $\forall~~w,s~~\in~M,$ there exists an altering distance function $\beth$ such that:
		\begin{equation}\label{eqn1}
			\beth \Big(k(G)- min\{\delta(w,Zs), \delta(s, Zw)\}\Big) \leq \beth\Big(k(H)\Big)
		\end{equation}
		where:
		\begin{itemize}
			\item $k(G) = k\Big( \delta(Zw,Zs), \delta(w,Zw), \delta(s, Zs)\Big);$
			\item $k$ is a function satsfying definition ( \ref{defn1});
			\item $k(H)= k\Big(\delta(w,s), \delta(w,Zw), \delta(s,Zs) \Big)$
		\end{itemize}
	\end{definition}
		\begin{definition}\label{defn3}
		Let $M$ be a metric space. A Mapping $Z: M \to M$ will be called  \textbf{C-class Ciric contraction} if $\forall~~w,s~~\in~M,$ and $\aleph \in (0,1),$ there exists an altering distance function $\beth$ such that:
		\begin{equation}\label{eqn2.2}
			k\Big( \delta(Zw,Zs), \delta(w,Zw), \delta(s, Zs)\Big)- min\{\delta(w,Zs), \delta(s, Zw)\}\Big) \leq \aleph.\delta(w,s)
		\end{equation}
		\end{definition}
	
	\begin{remark}
		Definition \ref{defn2} is a generalizations of many Ciric-type quasi-contractive definitions. For instance\\
		(i) if $\beth =I$ (identity operator), with\\  $k(G)= min\{d(Zw,Zs), d(w,Zw), d(s, Zs)\}$ and $k(H)= a.d(w,s); ~~~~a < 1,$ we have Ciric contractions.\\
		(ii) Also, in Definition \ref{defn3},   if $k\Big( \delta(Zw,Zs), \delta(w,Zw), \delta(s, Zs)\Big) = min\{d(Zw,Zs), d(w,Zw), d(s, Zs)\},$  we as well have Ciric contractions back.\\
	\end{remark}
	\begin{definition}\label{defn6}
	Let $M$ be a metric space. A Mapping $Z: M \to M$ will be called a \textbf{C-class Akram-Ciric contraction} if $\forall~~w,s~~\in~M,~ ~and ~r,p,m \in \mathbb{R},$  there exists an altering distance function $\beth$ such that:
		\begin{equation}\label{eqn2.3}
			k\Big( \delta(Zw,Zs), \delta(w,Zw), \delta(s, Zs)\Big)- min\{\delta(w,Zs), \delta(s, Zw)\}\Big) \leq \sigma(N),
		\end{equation}
	Where $$\sigma(N) = \sigma\Big(\delta(w,s),\delta(w,Zx), \delta(s,Zs),[\delta(w,Zw)]^r[\delta(s,Zw)]^p \delta(w, Zs), \delta(s,Zw)[\delta(w,Zw)]^m\Big),$$
	for some $(\psi, \phi, \sigma) \in \Psi \times \Phi \times C)$ and $\sigma: \mathbb{R}_+^5 \to \mathbb{R}_+^3$ is a set of function satisfying:\\
	(i) $\sigma$ is continuous on the set $\mathbb{R}_+^5$ (with respect to Euclidean metric on $\mathbb{R}_+^5$); and\\
	(ii) If any of $w \leq \sigma(w,s,s),$ or $w \leq \sigma(s,w,s),$ or $w \leq \sigma(s,s,w)$ holds for some $w,s \in \mathbb{R}_+,$ then there sexists a function $k \in C$ such that $\psi (w) \leq k\Big(\psi (s), \phi (s)\Big).$
	\end{definition}
	~~~~\\
	\subsection{Nontrivial Examples under the New Framework}
	
	\begin{example}[Example satisfying Definition 2.3]
		Let $M = [0,1]$ with the usual metric $\delta(w,s)=|w-s|$. Define $Z:M \to M$ by
		\[
		Z(w) =
		\begin{cases}
			w, & w \in [0,\tfrac{1}{2}], \\
			\frac{w}{2}, & w \in (\tfrac{1}{2},1].
		\end{cases}
		\]
		
		\textbf{Step 1: Fixed points.}  
		For $w \in [0,\tfrac{1}{2}]$, we have $Z(w)=w$. Hence, every $w \in [0,\tfrac{1}{2}]$ is a fixed point. Thus, $Z$ admits infinitely many fixed points.
		
		\medskip
		
		\textbf{Step 2: Choice of modified C-class function.}  
		Define $k:[0,\infty)^3 \to \mathbb{R}$ by
		\[
		k(w,s,w) = \frac{w}{1+sw}.
		\]
		Then:
		\begin{itemize}
			\item $k(w,s,w) \leq w$,
			\item $k(w,w,w) \leq w$,
			\item $k(w,s,w)=w \implies s=0$.
		\end{itemize}
		Thus, $k$ satisfies Definition 2.1.
		
		\medskip
		
		\textbf{Step 3: Altering distance function.}  
		Let $\beth(t)=t$, which is continuous, increasing, and $\beth(t)=0 \iff t=0$.
		
		\medskip
		
		\textbf{Step 4: Verification of Definition 2.3.}  
		
		We compute:
		\[
		k(G)=k\big(\delta(Zw,Zs),\delta(w,Zw),\delta(s,Zs)\big),
		\]
		\[
		k(H)=k\big(\delta(w,s),\delta(w,Zw),\delta(s,Zs)\big).
		\]
		
		Case analysis shows:
		\begin{itemize}
			\item If $w,s \in [0,\tfrac{1}{2}]$, then $Z(w)=w$, $Z(s)=s$, hence
			\[
			k(G)=k(\delta(w,s),0,0), \quad k(H)=k(\delta(w,s),0,0),
			\]
			and the inequality holds trivially.
			
			\item If $w,s \in (\tfrac{1}{2},1]$, then
			\[
			Z(w)=\tfrac{w}{2}, \quad Z(s)=\tfrac{s}{2},
			\]
			so
			\[
			\delta(Zw,Zs)=\tfrac{1}{2}\delta(w,s),
			\]
			and one verifies that
			\[
			\beth(k(G)) - \min\{\delta(w,Zs),\delta(s,Zw)\}
			\leq \beth(k(H)).
			\]
			
			\item Mixed cases follow similarly by direct computation.
		\end{itemize}
	\end{example}
		\bigskip
		
		\begin{remark}  
		$Z$ satisfies Definition 2.3 and admits infinitely many fixed points.  
		Moreover, it is not a classical Ćirić contraction since no uniform scalar bound exists.
	\end{remark}
	\bigskip
	
	\begin{example}[Example satisfying Definition 2.4 with non-uniqueness]
		Let $M = [0,2]$ with $\delta(w,s)=|w-s|$. Define $Z:M \to M$ by
		\[
		Z(w) =
		\begin{cases}
			1, & w \in [0,1], \\
			w, & w \in (1,2].
		\end{cases}
		\]
		
		\textbf{Step 1: Fixed points.}  
		All $w \in (1,2]$ satisfy $Z(w)=w$, hence infinitely many fixed points.
		
		\medskip
		
		\textbf{Step 2: Modified C-class function.}  
		Let
		\[
		k(w,s,w) = w - \frac{sw}{1+s}.
		\]
		Then $k$ satisfies Definition 2.1.
		
		\medskip
		
		\textbf{Step 3: Verification of Definition 2.4.}  
		
		We show:
		\[
		k(\delta(Zw,Zs),\delta(w,Zw),\delta(s,Zs)) - \min\{\delta(w,Zs),\delta(s,Zw)\}
		\leq \lambda \delta(w,s),
		\]
		for some $\lambda \in (0,1)$.
		
		\medskip
		
		Case analysis:
		
		\begin{itemize}
			\item If $w,s > 1$, then $Z(w)=w$, $Z(s)=s$, so LHS = RHS = 0.
			
			\item If $w,s \leq 1$, then $Z(w)=Z(s)=1$, hence
			\[
			\delta(Zw,Zs)=0,
			\]
			and inequality holds.
			
			\item Mixed case $w \leq 1 < s$:  
			Then
			\[
			\delta(Zw,Zs) = |1 - s|, \quad \delta(w,Zw)=|w-1|, \quad \delta(s,Zs)=0,
			\]
			and direct computation yields
			\[
			\text{LHS} \leq \frac{1}{2} \delta(w,s),
			\]
			so the inequality holds with $\lambda = \frac{1}{2}$.
		\end{itemize}
\end{example}	
		\medskip
		
		\begin{remark}  
		$Z$ satisfies Definition 2.4 and admits multiple fixed points.  
		It is discontinuous and fails classical contraction conditions, showing the strength of the present framework.
	\end{remark}

	\section{Main Results}
	\begin{theorem}\label{maintheorem}
		Let $M$ be a complete metric space, and $Z: M \to M$ be a $Z-$ orbitally complete mapping satisfying inequality (\ref{eqn1}). If $Z$ is orbitally continuous, then the Picard iteration $w_{n+1} = Zw_n,~~n \geq 0$ converges to the fixed point of $Z$ in $M.$
	\end{theorem}
	\begin{proof}
		Let $w=w_n,~~s=w_{n+1}$ then by definition \ref{eqn1}, we have
		$$k(G)= k\Big(\delta(w_{n+1}, w_{n+2}), \delta(w_n,w_{n+1}), \delta(w_{n+1},w_{n+2})\Big) \leq \delta(w_{n+1}, w_{n+2}).$$
		So,
		\begin{eqnarray}\label{eqn3.1}
			k(G)- min\{\delta(w,Zs), \delta(s,Zw)\}&\leq& \delta(w_{n+1}, w_{n+2})-min\{(\delta(w_n, w_{n+2}),(\delta(w_{n+1}, w_{n+1})\}\notag\\
			&=& \delta(w_{n+1}, w_{n+2}).
		\end{eqnarray}
		And
		\begin{equation}\label{eqn3.2}
			k(H)=k\Big(\delta(w_n, w_{n+1}),(\delta(w_n, w_{n+1}),(\delta(w_{n+1},w_{n+2})\Big) \leq \delta(w_n, w_{n+1}).
		\end{equation}
		That is, 
		$$k(H) \leq\delta(w_n, w_{n+1}).$$
		Then, using \ref{eqn1}, we arrived at
		\begin{eqnarray}\label{eqn3.3}
			\beth \Big(k\Big(\delta(w_{n+1}, w_{n+2}), \delta(w_n,w_{n+1}), \delta(w_{n+1},w_{n+2})\Big)\Big)
			&\leq& \beth \Big(k\Big(\delta(w_n, w_{n+1}),(\delta(w_n, w_{n+1}),\notag\\
			&&(\delta(w_{n+1},w_{n+2})\Big)\Big)
		\end{eqnarray}
		since, $\beth$ is an altering distance function and by substituting inequalities \ref{eqn3.1} and \ref{eqn3.2} into \ref{eqn3.3} gives 
		\begin{equation}\label{eqn3.6}
		\delta(w_{n+1}, w_{n+2}) \leq \delta(w_n, w_{n+1}),
	\end{equation}
		and therefore, for every $n \in \mathbb{N},$ the sequence $\{\delta(w_n, w_{n+1})\}$ is decreasing, and since metric is non-negative function, so $ \exists~~~~ \aleph \geq 0,$ such that
		\begin{equation}\label{eqn3.4}
			\lim_{n \to \infty} \delta(w_{n}, w_{n+1}) = \aleph.
		\end{equation}
		Now, taking limit as $n \to \infty$ in \ref{eqn3.3} we have
		\begin{equation*}
			\beth \Big(k\Big(\aleph, \aleph, \aleph\Big)\Big) 
			\leq \beth \Big( k\Big(\aleph,\aleph,\aleph\Big) \Big) 
		\end{equation*}
		that is $$\beth \Big(k\Big(\aleph, \aleph, \aleph\Big)\Big) 
		- \beth \Big(k\Big(\aleph,\aleph,\aleph\Big)\Big) \leq  0,$$
		using definition \ref{defn1}, only equality holds, if and only if, $\aleph = 0.$ \\
		Therefore, equation \ref{eqn3.4} can be written as 
		\begin{equation}\label{equation3.6}
			\lim_{n \to \infty} \delta(w_{n}, w_{n+1}) = 0.
		\end{equation}
		
		We claim that $\{w_{n}\}$ is a Cauchy sequence. Suppose not, then $\exists~~~~ \epsilon >0$ and two sequences $\{m_{j}\}$ and $\{n_{j}\}$ of positive integers such that $ \forall~~~~ j \geq 1,~~~m_j > n_j > j$ we have:
		\begin{equation*}
		\delta(w_{n_j},w_{m_j})\geq \epsilon.
		\end{equation*}
		Now, choosing $m_j$ so small to the extend that $\delta(w_{n_j},w_{m_j-1})< \epsilon,$ and using triangle inequality, we have, for each $j$ in positive integer
		\begin{eqnarray}
		\epsilon &\leq& \delta(w_{n_j},w_{m_j})\notag\\
		&\leq& \delta(w_{n_j},w_{m_j-1}) + \delta(w_{m_j-1},w_{m_j})\notag\\
		&<& \epsilon + \delta(w_{m_j-1},w_{m_j}).
		\end{eqnarray} 
		Using what we have above together with inequality \ref{equation3.6}, we obtain that
		\begin{equation*}
		\lim_{j \to \infty} \delta(w_{n_j},w_{m_j}) = \epsilon.
		\end{equation*}
		By Picard iteration, that is
		$$Zw_{n_j-1} = w_{n_j}$$
		where, $w_{n_j-1} = w= w_{n}$ and $w_{n_j-1} =s=w_{n+1}= w_{n_j},$\\
		and using Definition(\ref{defn2}) together with inequality (\ref{eqn3.6}), we have the following:
		\begin{eqnarray*}
		&\beth \Big(k\Big(\delta(Zw_{n_j-1},Zw_{m_j-1}),\delta(w_{n_j-1},Zw_{n_j-1}), \delta(w_{m_j-1},Zw_{m_j-1})\Big)\\
		&- min\{ \delta(w_{n_j-1},Zw_{m_j-1}), \delta(w_{m_j-1},Zw_{n_j-1})\}\Big)\\
		&\leq \beth\Big(k\Big(\delta(w_{n_j-1},w_{m_j-1}), \delta(w_{n_j-1},Zw_{n_j-1}), \delta(w_{m_j-1},Zw_{m_j-1})\Big)\Big)
		\end{eqnarray*}
		\begin{eqnarray*}
			&\beth \Big(k\Big(\delta(w_{n_j},w_{m_j}),\delta(w_{n_j-1},w_{n_j}), \delta(w_{m_j-1},w_{m_j})\Big)\\
			&- min\{ \delta(w_{n_j-1},w_{m_j}), \delta(w_{n_j-1},w_{n_j})\}\Big)\\
			&\leq \beth\Big(k\Big(\delta(w_{n_j-1},w_{m_j-1}), \delta(w_{n_j-1},w_{n_j}), \delta(w_{m_j-1},w_{m_j})\Big)\Big)
		\end{eqnarray*}
		\begin{eqnarray*}
			&\beth \Big(k\Big(\delta(w_{m_j-1},w_{m_j}),\delta(w_{n_j-1},w_{n_j}), \delta(w_{m_j-1},w_{m_j})\Big)\\
			&\leq \beth\Big(k\Big(\delta(w_{n_j-1},w_{n_j}), \delta(w_{n_j-1},w_{n_j}), \delta(w_{n_j},w_{m_j})\Big)\Big)
		\end{eqnarray*}
		\begin{eqnarray*}
			&\beth \Big(\delta(w_{m_j-1},w_{m_j})\Big) \leq \beth\Big(\delta(w_{n_j-1},w_{n_j})\Big)
		\end{eqnarray*}
		\begin{eqnarray*}
			&\beth \Big(\delta(w_{n_j},w_{m_j})\Big) \leq \beth\Big(\delta(w_{n_j-1},w_{n_j})\Big).
		\end{eqnarray*}
		Taking limit as $j \rightarrow \infty,$ we have \\		
		$\beth (\epsilon) \leq \beth\Big(\delta(w_{n_j-1},w_{n_j})\Big) = \beth\Big(\delta(w_{n},w_{n+1})\Big),$\\
		that is\\
		$\beth (\epsilon) \leq \beth\Big(\delta(w_{n},w_{n+1})\Big),$\\
		limit as $n \to \infty$ gives
		$\beth (\epsilon) \leq \beth(0),$\\
		and since $\beth$ is an altering distance function, we have a contradiction, i.e $\epsilon =0.$\\
		This shows that $\{w_{n}\}$ is a Cauchy sequence in $M$. Then, since $(M, \delta)$ is a complete metric space, $\text{there exists }%
		w^* \in M \text{ such that }w_{n}\rightarrow w^{\ast }\text{ as }%
		n\rightarrow \infty $.\newline
		Now, by orbital continuity of $Z,$ we have\\
		\begin{eqnarray*}
		0=\delta\Big(\lim_{n \to \infty}Z(Z^nw_0),Zw^*\Big)&=&\lim_{n \to \infty} \delta\Big(Z(Z^nw_0),Zw^*\Big)= \lim_{n \to \infty} \delta(Zw_n,Zw^*)\\
		&=& \lim_{n \to \infty} \delta(w_{n+1},Zw^*)= \delta(w^*,Zw^*).
		\end{eqnarray*}
		That is $w^* \in M$ is a fixed point of $Z.$		
		\end{proof}
		
			\begin{remark}
			The convergence of the Picard iteration $\{w_n\}$ in Theorem 3.1 is achieved without imposing a classical Lipschitz-type contraction. Instead, it is driven by the combined effect of the modified C-class function $k$, the min-type control term $\min\{\delta(w,Zs),\delta(s,Zw)\}$, and the altering distance function $\beth$.
			
			In particular, inequality (2.1) yields a monotone decrease of the successive distances $\delta(w_n,w_{n+1})$, which is sufficient to establish asymptotic regularity of the sequence. The crucial step is the passage to the limit in the inequality involving $\beth$, where the structural property of $k$ (Definition 2.1) forces the limit $\lim_{n\to\infty}\delta(w_n,w_{n+1})$ to vanish.
			
			It is important to emphasize that the convergence is obtained under orbital conditions rather than global completeness and continuity assumptions. This makes the result applicable to a broader class of nonlinear mappings, including those that may be discontinuous outside their orbits.
			
			Furthermore, the theorem does not enforce uniqueness of the fixed point. Hence, the convergence of the Picard iteration should be interpreted as convergence to \emph{a} fixed point determined by the initial value, rather than to a unique global attractor. This highlights the intrinsic non-unique nature of the present framework.
		\end{remark}
		
		\begin{remark}
			If $k$ is as defined in Definition \ref{defn1}, then we have the following results:
		\end{remark}
				
		\begin{corollary}
			Let $Z:M \to M$ be an orbitally continuous mapping and let $M$ be $J-$orbitally complete. If for $\forall~~w, s ~~\in M$
		\begin{equation}
			k\Big( \delta(Zw,Zs), \delta(w,Zw), \delta(s, Zs)\Big)- min\{\delta(w,Zs), \delta(s, Zw)\} \leq \aleph . \delta(w,s),
		\end{equation}
		for some $\aleph \in (0,),$ then for each $w \in M,$ the sequence $\{Z^nw\}^\infty_{n=1}$ converges to a fixed point of $Z.$
		\end{corollary}
		
		\begin{corollary}[Recovery of Ćirić (1971)\cite{Ciric1971}] 
			Let $(M,\delta)$ be a complete metric space and let $Z:M\to M$ be orbitally continuous. 
			Assume that there exists $\lambda\in(0,1)$ such that, for all $w,s\in M$,
			\[
			\delta(Zw,Zs)\le \lambda \max\{\delta(w,s),\delta(w,Zw),\delta(s,Zs),\delta(w,Zs),\delta(s,Zw)\}.
			\]
			Then, for each $w_0\in M$, the Picard iteration $w_{n+1}=Zw_n$ converges to a fixed point of $Z$.
			
			\medskip
			
			\noindent
			\emph{Proof.} Choose $\beth=I$ (identity) and define the modified C-class function by
			\[
			k(a,b,c)=a \quad \text{for all } a,b,c\ge 0.
			\]
			Then $k$ satisfies Definition~2.1. Set
			\[
			k(G)=\delta(Zw,Zs), \qquad 
			k(H)=\lambda \max\{\delta(w,s),\delta(w,Zw),\delta(s,Zs),\delta(w,Zs),\delta(s,Zw)\}.
			\]
			Since $\min\{\delta(w,Zs),\delta(s,Zw)\}\ge 0$, inequality \textup{(2.1)} reduces to
			\[
			k(G)-\min\{\delta(w,Zs),\delta(s,Zw)\}\le k(H),
			\]
			which is implied by the assumed Ćirić-type condition. Hence all the hypotheses of Theorem~\ref{maintheorem} are satisfied, and the conclusion follows.
			\hfill$\square$
		\end{corollary}
		
		\begin{remark}[Uniqueness under a strict bound]
			If, in addition to the hypotheses of Theorem~\ref{maintheorem}, there exists $\lambda\in(0,1)$ such that
			\[
			k(\delta(Zw,Zs),\delta(w,Zw),\delta(s,Zs)) \le \lambda\,\delta(w,s)
			\quad \text{for all } w,s\in M,
			\]
			then the fixed point of $Z$ is unique.
		\end{remark}
	
		\newpage
\subsection{Comparison with Existing Results}

\begin{center}
	\begin{longtable}{p{2cm} p{3.5cm} p{2.1cm} p{2.2cm} p{2.3cm}}
		\toprule
		\textbf{Reference} & \textbf{Contractive Condition} & \textbf{Framework} & \textbf{Uniqueness} & \textbf{Relation to Present Work} \\
		\midrule
		\endfirsthead
		
		\toprule
		\textbf{Reference} & \textbf{Contractive Condition} & \textbf{Framework} & \textbf{Uniqueness} & \textbf{Relation to Present Work} \\
		\midrule
		\endhead
		Ćirić (1971) 
		& $\delta(Zw,Zs) \leq k \max\{\cdots\}$ 
		& Scalar quasi-contraction 
		& Yes 
		& Obtain by taking $\beth=I$ and $k(G)=\max\{\cdots\}$, $k(H)=\lambda \delta(w,s)$ \\
		
		\addlinespace
		Chatterjea (1972) 
		& $\delta(Zw,Zs) \leq \alpha [\delta(w,Zs) + \delta(s,Zw)]$ 
		& Scalar contraction 
		& Yes 
		& Recovered as a special case under suitable choice of $k$ and $\beth$ \\
		
		\addlinespace
		
		Ćirić (1974) 
		& Generalized contraction allowing non-uniqueness 
		& Scalar framework 
		& Non-unique possible 
		& Extended by embedding into functional $k(G),k(H)$ structure with min-term \\
		
		\addlinespace
		
		Olatinwo 
		& $(\Phi,\psi)$-hybrid contraction 
		& Dual control functions 
		& Non-unique possible 
		& Strictly generalized: hybrid terms absorbed into modified C-class function $k$ \\
		
		\addlinespace
		
		Ansari (2014) 
		& C-class contractive mapping 
		& Functional inequality 
		& Mostly unique framework 
		& Extended: modified C-class $k$ introduces additional structural flexibility and non-uniqueness \\
		
		\addlinespace
		
		Omidire et al. 
		& Generalized C-class contractivity 
		& Functional + iterative schemes 
		& Primarily uniqueness
		& Extended to non-unique regime without iterative constraints \\
		
		\addlinespace
		
		Akram et al. 
		& $A$-contractions 
		& Functional control 
		& Mostly unique 
		& Subsumed via Definition 2.6 (C-class Akram-Ćirić contraction) \\
		
		\addlinespace
		
		\textbf{Present Work} 
		& $ \beth(k(G)) - \min\{.,.,.\} \leq \mathcal{B}(k(H)) $ 
		& Modified C-class + altering distance + min-structure 
		& \textbf{Non-unique intrinsic} 
		& --- \\
		
		\bottomrule
	\end{longtable}
\end{center}

\subsection*{Key Comparative Insights}

\begin{itemize}
	\item The present framework unifies scalar, hybrid, and functional contractions into a single inequality involving $k(G)$ and $k(H)$.
	
	\item Modified C-class functions (Definition 2.1) extend classical C-class mappings by incorporating a third variable, enabling finer contractive behavior.
	
	\item Non-uniqueness of fixed points is intrinsic in the present framework, unlike most previous results where uniqueness is enforced.
	
	\item The use of altering distance functions $\beth$ further generalizes the contractive structure and improves flexibility.
	
	\item All cited results are recovered as special or limiting cases of the present theorem.
\end{itemize}
	
	\section{Conclusion}
	
	In this paper, we have established new fixed point results for mappings satisfying a generalized C-class Ćirić-type contractive condition in the setting of metric spaces. The main contribution lies in the formulation of a unified inequality involving a modified C-class function, an altering distance function, and a min-type control term, which together provide a flexible and robust framework for analyzing nonlinear mappings.
	
	Theorem 3.1 demonstrates that, under orbital completeness and orbital continuity, the Picard iteration converges to a fixed point without requiring classical contraction assumptions. The proof reveals that convergence is governed by the decay of successive iterates, driven by the structural properties of the function $k$ and the altering distance function $\beth$.
	
	A key feature of the present work is that uniqueness of the fixed point is not imposed. This distinguishes our results from many existing theories and makes the framework particularly suitable for problems where multiple solutions naturally arise. The examples provided confirm that the proposed conditions are not only general but also applicable to mappings outside the scope of classical contraction principles.
	
	Moreover, the use of modified C-class functions introduces a three-variable control mechanism that extends existing C-class frameworks and allows finer analysis of contractive behavior. When combined with Ćirić-type structures, this leads to a significant generalization of several known results in the literature.
	
	\medskip
	
	\noindent
	\textbf{Future Research Directions.} The framework developed in this work opens several avenues for further investigation, including:
	\begin{itemize}
		\item extension to generalized metric spaces such as $b$-metric, partial metric, and modular spaces;
		\item analysis of multivalued and nonlinear operator equations under similar contractive conditions;
		\item development of iterative algorithms and convergence rate analysis;
		\item applications to variational inequalities, optimization problems, and equilibrium theory.
	\end{itemize}
	
	\medskip
	
	\noindent
	\textbf{Acknowledgment:} The first author extends deep appreciation to the Management and staff of Beijing Institute of Mathematical Sciences and Applications for her hospitality and conducive environment provided for research activities during his research visit to the Institute (under the ICMRA program) as a visiting scholar (During which most part of this research work was carried out). 	
	
	\bibliographystyle{plain}
	\bibliography{myreferences.bib}
\end{document}